\documentclass[11pt]{article}
\usepackage[margin = 1in]{geometry}
\usepackage{setspace}
\usepackage{hyperref}\usepackage{enumitem}
\setlist[enumerate]{label=$(\alph*)$, leftmargin = *}
\usepackage{amsmath,amsthm,amssymb,stmaryrd,wasysym, xcolor}
\usepackage{tikz-cd}
\usepackage{authblk}
\providecommand{\keywords}[1]
{
  {\small	
  \textbf{\textit{Keywords:}}} #1
}
\theoremstyle{plain}
\newtheorem{theorem}{Theorem}[section]
\newtheorem{corollary}[theorem]{Corollary}
\newtheorem{lemma}[theorem]{Lemma}

\newtheorem{proposition}[theorem]{Proposition}
\theoremstyle{definition}

\newtheorem{definition}[theorem]{Definition}

\newtheorem{example}[theorem]{Example}

\newcommand{\cV}{\mathcal V}

\newcommand{\NN}{{\mathbb N}}

\newcommand{\vmon}{{\cV\text{-}\bf Mon}}
\newcommand{\twomon}{{{\bf 2}\text{-}\bf Mon}}
\newcommand{\vcat}{{\cV\text{-}\bf Cat}}

\newcommand{\mon}{{\bf Mon}}

\newcommand{\two}{{\bf 2}}

\newcommand{\llex}{{{\rm lex}}}
\newcommand{\wlex}{{{\rm wlex}}}

\begin{document}
\title{On semidirect products of quantale enriched monoids}\date{} \author{C\'elia Borlido\footnote{Partially supported by the
    Centre for Mathematics of the University of Coimbra -
    UIDB/00324/2020, funded by the Portuguese Government through
    FCT/MCTES (doi: 10.54499/UIDB/00324/2020).}}\affil{Centre for
  Mathematics, University of
  Coimbra (CMUC)\\Department of Mathematics, 3001-501 Coimbra, Portugal \\~\\
  cborlido@mat.uc.pt}\maketitle
\begin{abstract}
  We consider monoids equipped with a compatible quantale valued
  relation, to which we call quantale enriched monoids, and study
  semidirect products of such structures. It is well-known that
  semidirect products of monoids are closely related to Schreier split
  extensions which, in the setting of monoids, play the role of split
  extensions of groups.  We will thus introduce certain split
  extensions of quantale enriched monoids, which generalize the
  classical Schreier split extensions of monoids, and investigate
  their connections with semidirect products. We then restrict our
  study to a class of quantale enriched monoids whose behavior mimics
  the fact that the preorder on a preordered group is completely
  determined by its cone of positive elements. Finally, we instantiate
  our results for preordered monoids and compare them with existing
  literature.
\end{abstract}

\section{Introduction}

In the theory of groups, the study of split extensions plays an
important role as it allows for the decomposition of each group as a
semidirect product of each of its normal subgroups and corresponding
quotients.  In recent years, there had been generalizations of this
result in two different directions.

On the one hand, there were considered
in~\cite{ClementinoMartins-FerreiraMontoli19, ClementinoRuivo2023}
groups equipped with a preorder compatible with the multiplication
(but not with the inversion!) and split extensions of preordered
groups were characterized. Later, the theory was further generalized
to groups enriched in a quantale~$\cV$~\cite{ClementinoMontoli21}, the
so-called $\cV$-groups, thereby obtaining results that apply not only
to preordered groups but also to various structures such as
generalized (ultra)metric groups and probabilistic (ultra)metric
groups.

On the other hand, there were considered
in~\cite{Martins-FerreiraMontoliSobral13} split extensions of
monoids. In \emph{loc. cit.}, it was shown that not every split
extension of monoids would give rise to a semidirect product but one
should instead restrict to the so-called \emph{Schreier split
  extensions}. In the case where the structures at play are groups,
the concepts of Schreier split extension and of split extension
coincide. Once again, Schreier split extensions of monoids were later
explored in an enriched setting by considering monoids equipped with a
preorder compatible with the
multiplication~\cite{Martins-FerreiraSobral21}. We note however that
the results of~\cite{Martins-FerreiraSobral21} are not a
generalization of those
of~\cite{ClementinoMartins-FerreiraMontoli19,ClementinoRuivo2023,
  ClementinoMontoli21}. For one, the definition of \emph{Schreier
  split extension} of~\cite{Martins-FerreiraSobral21} is not a
generalization of the definition of \emph{split extension}
of~\cite{ClementinoMartins-FerreiraMontoli19,ClementinoRuivo2023} for
preordered groups (and thus, neither of~\cite{ClementinoMontoli21} for
$\cV$-groups). Moreover, the focus
of~\cite{ClementinoMartins-FerreiraMontoli19,ClementinoRuivo2023}
and~\cite{ClementinoMontoli21} is different of that
of~\cite{Martins-FerreiraSobral21}. While in the former the authors
characterize the preorders (respectively, $\cV$-group structures) on a
given semidirect product of groups that determine a split extension of
preodered groups (respectively, of $\cV$-groups), in the latter the
authors abstractly provide a characterization of (certain) Schreier
split extensions of preordered monoids without paying attention to the
preorders that may possibly occur.

In this paper we study split extensions of quantale-enriched monoids,
thereby generalizing both approaches at once. After briefly recalling,
in Section~\ref{sec:prel}, the most relevant concepts and results used
in the remaining paper, we introduce, in Section~\ref{sec:v-mon}, the
category of $\cV$-monoids, that is, of \emph{monoids enriched in a
  quantale $\cV$}. Split extensions of $\cV$-monoids are considered in
Section~\ref{sec:schreier} where we introduce the notion of
\emph{$U$-Schreier split extension}, for a suitable functor $U$, and
characterize the $\cV$-monoid structures that may possibly appear in
the semidirect product component of a $U$-Schreier split
extension. Section~\ref{sec:vmon-star} is then devoted to studying
those $\cV$-monoids whose quantale enrichment is determined by a
suitable analogue of the positive cone for preordered monoids, as it
happens for $\cV$-groups. Finally, in Section~\ref{sec:preorderd-mon},
we instantiate the results of~Section~\ref{sec:vmon-star} for
preordered groups and compare them with those
of~\cite{Martins-FerreiraSobral21}.

\section{Preliminaries}\label{sec:prel}
The reader is assumed to have some acquaintance with basic category
theory and semidirect products of monoids. With the aim of setting up
the notation, we recall the most relevant concepts that will be used
in the paper. For more on general category theory, including the
missing definitions, we refer to~\cite{MacLane1998}, for
$\cV$-categories to~\cite{HofmannSealTholen2014}, and for Schreier
split extensions and semidirect products of monoids
to~\cite{Martins-FerreiraMontoliSobral13}.
\subsection{$\cV$-categories}

A \emph{quantale} is a tuple $\cV = (V, \le, \otimes, k)$ such that
$(V, \le)$ is a complete lattice, $(V, \otimes, k)$ is a monoid, and
the following equalities hold:
\begin{equation}
  a \otimes (\bigvee_i b_i) = \bigvee_i{(a \otimes b_i)} \qquad
  \text{and} \qquad (\bigvee_i a_i) \otimes b = \bigvee_i(a_i \otimes
  b).\label{eq:7}
\end{equation}
The top and bottom elements of $(V, \leq)$ will be denoted by $\top$
and $\bot$, respectively.  Notice that~\eqref{eq:7} implies that
$\otimes$ is monotone with respect to the partial order on $V$, that
is, if $v_1 \le v_2$ and $w_1 \le w_2$, then $v_1\otimes w_1 \le v_2
\otimes w_2$. We say that $\cV$ is \emph{commutative} provided the
monoid $(V, \otimes, k)$ is commutative.  The quantale $\two = (\{0,
1\}, \leq, \wedge, 1)$, where $0 \leq 1$, will be important in this
work, namely when considering preordered monoids on
Section~\ref{sec:preorderd-mon}.

Let $\cV$ be a quantale, and let $X$, $Y$ be two sets. A
\emph{$\cV$-relation} from $X$ to $Y$ is a function $a: X \times Y \to
V$.
A \emph{$\cV$-category} is a pair $(X, a)$, where $X$ is a set and $a$
is a reflexive and transitive $\cV$-relation from $X$ to $X$, that is,
for every $x, y, z \in X$, we have
\begin{itemize}
\item $k \le a(x,x)$ (reflexivity);
\item $a(x,y) \otimes a(y, z) \le a(x, z)$ (transitivity).
\end{itemize}
A \emph{$\cV$-functor} from $(X, a)$ to $(Y, b)$ is a set function $f:
X\to Y$ such that, for all $x_1, x_2 \in X$,
\[a(x_1, x_2) \le b(f(x_1), f(x_2)).\]
We denote by $\vcat$ the category whose objects are $\cV$-categories
and morphisms are $\cV$-functors. Note that a $\two$-category may be
simply seen as a preordered set, while a morphism of $\two$-categories
is a monotone function.

When $\cV$ is a commutative quantale, for $\cV$-categories $(X, a)$
and $(Y, b)$ we may define a new $\cV$-category $(X \times Y, a
\otimes b)$, where
\[(a \otimes b)((x, y), (x', y')) = a(x, x') \otimes b(y,y'),\]
for all $x, x' \in X$ and $y, y' \in Y$~\cite[Proposition
III.1.3.3]{HofmannSealTholen2014}. As this is an essential
construction in this paper, \emph{all} the quantales we consider will
be commutative, even if we do not mention it explicitly.

\subsection{Semidirect products of monoids}

In the theory of groups, it is a well-known result that, for groups
$H$ and $N$, semidirect products of the form $N \rtimes H$ are in a
bijective correspondence with split exact sequences $N \to G \to
H$. This result has later been extended to preordered
groups~\cite{ClementinoMartins-FerreiraMontoli19} and further
generalized to $\cV$-groups~\cite{ClementinoMontoli21}. On the other
hand, in the context of monoids, if one aims at describing semidirect
products of monoids via suitable short exact sequences, then one
should restrict to the so-called \emph{Schreier split
  extensions}~\cite{Martins-FerreiraMontoliSobral13}.

In what follows, we fix three monoids $(X, +)$, $(Y, \cdot)$, and $(Z,
\star)$. For the sake of readability we denote the operation on~$X$
additively, though $X$ is not assumed to be commutative.
\begin{definition}
  A \emph{Schreier point of monoids} is a split epimorphism $p: Z \to
  Y$, together with a section $s: Y \to Z$, for which there exists a
  unique set map $q: Z \to X$ satisfying
  \begin{equation}
    z = kq(z) \star sp(z),\label{eq:2}
  \end{equation}
  where $k : X \to Z$ is the kernel of~$p$.

  A \emph{Schreier split extension of monoids} is a diagram of the
  form
  \begin{equation}
    X \overset{k} {\hookrightarrow} Z
    \overset{p}{\underset{s}\rightleftarrows} Y,\label{eq:19}
  \end{equation}
  where $p$ is a Schreier point of monoids with section~$s$, and $k$
  the kernel of~$p$.
\end{definition}

We have that every Schreir split extension of monoids is a split exact
sequence~\cite[Proposition~2.7]{Martins-FerreiraMontoliSobral13}. For
a Schreier split extension as in~\eqref{eq:19}, we will often denote
by~$q$ the unique set map satisfying~\eqref{eq:2}.

Given a function $\alpha: Y \times X \to X$, we define a binary
operation on $X \times Y$ by
\[(x_1, y_1)\star (x_2, y_2) = (x_1 + \alpha(y_1, x_2), y_1\cdot
  y_2).\]
The algebra thus obtained is denoted by $X \rtimes_\alpha Y$.  Note
that $\alpha$ may be recovered from the binary operation on $X
\rtimes_\alpha Y$. Indeed, for every $x \in X$ and $y \in Y$, we have
\[(0, y)\star (x, 1) = (\alpha(y, x), y).\]
It is well-known that $\alpha$ is a monoid action if, and only if, $(X
\rtimes_\alpha Y, \star)$ is a monoid.

\begin{theorem}[{\cite[Theorem
    2.9]{Martins-FerreiraMontoliSobral13}}]\label{t:1} 
  There is a one-to-one correspondence between the Schreier split
  extensions $X \overset{k} {\hookrightarrow} Z
  \overset{p}{\underset{s}\rightleftarrows} Y$ and the monoid actions
  of~$Y$ on~$X$ (and thus, with the semidirect products $(X
  \rtimes_\alpha Y, \star)$). More precisely,
  \begin{enumerate}
  \item if $X \overset{k} {\hookrightarrow} Z
    \overset{p}{\underset{s}\rightleftarrows} Y$ is a Schreir split
    extension of monoids and $q: Z \to X$ is the unique set map
    satisfying~\eqref{eq:2}, then $\alpha: Y \times X \to X$ defined
    by $\alpha(y, x) = q(s(b) \star k(x))$ is a monoid action,
  \item if $(X \rtimes_\alpha Y, \star)$ is a semidirect product of
    monoids, then $X \overset{\iota_1} {\hookrightarrow} X
    \rtimes_\alpha Y
    \overset{\pi_2}{\underset{\iota_2}\rightleftarrows} Y$ is a
    Schreir split extension of monoids whose unique set map
    satisfying~\eqref{eq:2} is the projection $\pi_1: X \rtimes_\alpha
    Y \to X$.
  \end{enumerate}
  Moreover, these two assignments are mutually inverse and, if $X
  \overset{k} {\hookrightarrow} Z \overset{p}{\underset{s}\rightleftarrows} Y$
  is a Schreir split extension of monoids, then the maps
  \begin{equation}
    \varphi: X \rtimes_\alpha Y \to Z, \qquad (x, y) \mapsto k(x)
    \star s(y)\label{eq:18}
  \end{equation}
  and  
  \begin{equation}
    \psi: Z \to X \rtimes_\alpha Y, \qquad z \mapsto (q(z),
    p(z))\label{eq:17}
  \end{equation}
  are mutually inverse monoid homomorphisms.
\end{theorem}

The preordered version of Theorem~\ref{t:1} was considered
in~\cite{Martins-FerreiraSobral21}, but only for a special class of
preordered monoids that retain a certain group-like behavior. Further
details on this will be provided in Section~\ref{sec:preorderd-mon}.
\section{The category of quantale enriched monoids}\label{sec:v-mon}

Let $\cV$ be a (commutative) quantale. A \emph{$\cV$-monoid} is a
triple $(X, a, +)$ such that $(X, a)$ is a $\cV$-category, $(X, +)$ is
a monoid and the monoid operation induces a $\cV$-functor
$(\_+\_):(X,a) \otimes (X, a) \to (X, a)$. We recall that, although we
are denoting the monoid operation additively, we do not assume that
$X$ is commutative and this will be the usual practice in the
remaining paper. When the monoid operation is clear from the context
(or irrelevant), we may simply say that $(X, a)$ is a $\cV$-monoid. A
\emph{morphism of $\cV$-monoids} $h: (X, a, +) \to (Y, b, \cdot)$ is a
set map $h:X \to Y$ such that $h: (X, a) \to (Y, b)$ is a
$\cV$-functor and $h: (X, +) \to (Y, \cdot)$ is a monoid
homomorphism. We denote by $\cV$-$\mon$ the category of $\cV$-monoids
and corresponding homomorphisms.

Starting from the category of $\cV$-monoids, we may either forget the
$\cV$-category or the monoid structure, thereby obtaining two
forgetful functors $U: \vmon \to \mon$ and $V:\vmon \to \vcat$,
respectively. Similarly to what happens for
$\cV$-groups~\cite{ClementinoMontoli21} and for preordered
monoids~\cite{Martins-FerreiraSobral21}, $U$ is a topological functor
and $V$ a monadic one. We do not include the proofs of these two
facts, as they are simple adaptations of the proof
of~\cite[Theorem~4.1]{ClementinoMontoli21}. In particular, it follows
that the category of $\cV$-monoids is both complete and
cocomplete. Moreover, limits are preserved by both forgetful functors,
while colimits are preserved by~$U$.  We make a few observations that
will be relevant in the sequel. First note that the initial object of
$\vmon$ is $(\{*\}, \kappa)$, where $\kappa(*, *) = k$, while its
terminal object is $(\{*\}, \tau)$, where $\tau(*,*) = \top$. In
particular, $\vmon$ is a pointed category if, and only if, $k = \top$.
When that is the case, the kernel of a morphism $h: (X, a) \to (Y, b)$
of $\cV$-monoids is the $\cV$-monoid $(Z, c)$, where
\[Z = \{x \in X\mid h(x) = 1\}\]
is a submonoid of $X$ and $c$ is the suitable restriction of
$a$. Finally, an epimorphism of $\cV$-monoids is simply a morphism
whose underlying monoid homomorphism is an epimorphism.

The following result is a simple observation, but we will occasionally
use it in the remaining paper.
\begin{lemma}\label{l:8}
  Let $(X, a)$ be a $\cV$-category and $(X, +)$ be a monoid. Then,
  $(\_+\_):(X, a) \otimes (X, a) \to (X, a)$ is a $\cV$-functor if,
  and only if, for every $x, y, z \in X$,
  \begin{equation}
    a(x, y) \leq a(x + z, y + z) \quad \text{and}\quad a(x, y) \leq
    a(z + x, z + y).\label{eq:4}    
  \end{equation}
\end{lemma}
\begin{proof}
  By definition, $+$ is a $\cV$-functor if, and only if, for all $x_1,
  x_2, y_1, y_2 \in X$, the following equality holds
  \[a(x_1,y_1) \otimes a(x_2, y_2) \leq a(x_1+ x_2, y_1+ y_2).\]
  In particular, using the fact that $a$ is reflexive, we have
  \[a (x, y) = a(x, y) \otimes k \leq a(x, y) \otimes a(z, z) \leq a(x
    + z, y + z)\]
  and
  \[a (x, y) = k \otimes a(x, y) \leq a(z, z) \otimes a(x, y)\leq a(z
    + x, z + y).\]
  Conversely, using~\eqref{eq:4} and transitivity of~$a$, we may
  derive that
  \[a(x_1,y_1) \otimes a(x_2, y_2) \leq a(x_1 + x_2,y_1 + x_2) \otimes
    a(y_1+ x_2, y_1+ y_2) \leq a(x_1+ x_2, y_1+ y_2).\popQED\]
\end{proof}

\section{Semidirect products of quantale enriched
  monoids}\label{sec:schreier}

In this section we will assume that the quantale $\cV$ is such that
$\vmon$ is a pointed category, that is, $k = \top$ in $\cV$. We will
also fix $\cV$-monoids $(X, a, +)$, $(Y, b, \cdot)$, and $(Z, c,
\star)$.

As already mentioned, semidirect products of monoids are closely
related to the so-called Schreier split extensions of monoids. In the
context of quantale enriched monoids, we shall consider the following
definitions. Recall that we have a forgetful functor $U: \vmon \to
\mon$.
\begin{definition}
  We call \emph{$U$-Schreier point of $\cV$-monoids} to a split
  epimorphism of $\cV$-monoids $p: (Z, c, \star) \to (Y, b, \cdot)$,
  together with a section $s$, such that $Up: UZ \to UY$, together
  with the section $Us$, is a Schreier point of monoids.

  A \emph{$U$-Schreier split extension of $\cV$-monoids} is a diagram
  of the form
  \[ (X, a) \overset{k} {\hookrightarrow} (Z, c)
    \overset{p}{\underset{s}\rightleftarrows} (Y, b)\]
  in $\vmon$, where $p$ is a $U$-Schreier point of $\cV$-monoids with
  section $s$, and $k$ is the kernel of $p$.
\end{definition}

Our goal is to present a characterization of the $U$-Schreier points
of $\cV$-monoids with codomain $(Y, b)$ and kernel $(X, a)$ or, in
other words, the $U$-Schreier split extensions of $\cV$-monoids of the
form
\begin{equation}
  (X, a) \overset{k} {\hookrightarrow} (Z, c)
  \overset{p}{\underset{s}\rightleftarrows} (Y, b).\label{eq:3}
\end{equation}
We remark that, since the diagram formed by the underlying monoid
homomorphisms of~\eqref{eq:3} is a Schreier split extension of
monoids, by Theorem~\ref{t:1}, the monoid $Z$ is isomorphic to a
semidirect product of the form $X \rtimes_\alpha Y$, the maps $k$ and
$s$ are isomorphic to the inclusions $\iota_1$ and $\iota_2$,
respectively, and $p$ is isomorphic to the projection
$\pi_2$. Moreover, the unique set map~$q$ satisfying~\eqref{eq:2} is
isomorphic to the projection $\pi_1: X \rtimes_\alpha Y \to X$. In
particular, the diagram~\eqref{eq:3} is isomorphic to
\begin{equation}
  (X, a) \overset{\iota_1} {\hookrightarrow} (X \rtimes_\alpha Y, c)
  \overset{\pi_2}{\underset{\iota_2}\rightleftarrows} (Y,
  b)\label{eq:8}
\end{equation}
and the latter is a $U$-Schreier split extension of $\cV$-monoids if, and
only if, for all $x, x' \in X$ and $y,y' \in Y$, the following
conditions hold:
\begin{enumerate}[label = (S.\arabic*)]
\item\label{item:S1} $\pi_2$ is a $\cV$-functor, that is,
  \[c((x, y), (x', y')) \leq b(y, y');\]
\item\label{item:S2} $\iota_1$ is the kernel of $\pi_2$, that is,
  \[a(x, x') = c((x, 1), (x', 1));\]
\item\label{item:S3} $\iota_2$ is a $\cV$-functor, that is,
  \[b(y, y') \leq c((0, y), (0, y')).\]
\end{enumerate}

In~\cite{ClementinoMontoli21}, for $\cV$-monoids $(X, a)$ and $(Y,
b)$, the (reverse) lexicographic $\cV$-relation $\llex: (X \times Y)
\times (X \times Y) \to \cV$ was defined by
\[ \llex((x, y), (x',y')) =
  \begin{cases}
    a(x, x'), \text{ if } y = y'; \\ b(y, y'), \text{ else}.
  \end{cases}\]
Here, we consider its weaken version $\wlex: (X \times Y) \times (X
\times Y) \to V$ given by
\[ \wlex((x, y), (x',y')) =
  \begin{cases}
    a(x, x'), \text{ if } y = y' = 1;
    \\ b(y, y'), \text{ else}.
  \end{cases}\]
We note that the $\cV$-relations $\llex$ and $\wlex$ coincide on every
tuple $((x, y), (x',y'))$ unless $y = y' \neq 1$ and, in that case, we
have
\[\llex ((x, y), (x', y)) = a(x, x') \leq k = b(y, y) = \wlex ((x,
  y), (x', y)).\] This in particular shows that $\llex \leq \wlex$. In
fact, the two relations coincide only in very particular cases.

\begin{lemma}\label{l:1}
  Let $(X, a)$ and $(Y, b)$ be $\cV$-monoids. Then, the following are
  equivalent:
  \begin{enumerate}
  \item \label{item:7} $\llex = \wlex$,
  \item\label{item:8} $Y$ is trivial or $a(x, x') = k$ for all $x,
    x'\in X$.
  \end{enumerate}
\end{lemma}
\begin{proof}
    We first observe that $\llex = \wlex$ if, and only if, for every $x,
  x' \in X$ and $y \in Y \setminus \{1\}$, the equality
  \[\llex((x, y), (x', y)) = \wlex((x, y), (x', y))\]
  holds.
  By definition of $\llex$ and $\wlex$ this is equivalent to having
  \[a(x, x') = b(y, y) = k,\]
  from where we may conclude that \ref{item:7} and \ref{item:8} are
  indeed equivalent.
\end{proof}

We also note that the $\cV$-relations $\llex$ and $\wlex$ are always
reflexive but they may not be transitive. Indeed, we have the
following:
\begin{lemma}\label{l:3}
  The $\cV$-relation $\wlex$ is transitive if, and only if, for every
  $x, x' \in X$ and every $y \in Y \setminus \{1\}$, we have $b(1,
  y) \otimes b(y, 1) \leq a(x, x')$.
\end{lemma}
\begin{proof} If the $\cV$-relation $\wlex$ is transitive then, for $y
  \neq 1$, we have
  \begin{align*}
    b(1, y) \otimes b(y, 1)
    & = \wlex ((x, 1), (0, y)) \otimes \wlex((0, y), (x',1)) \qquad
      \text{(by definition of $\wlex$)}
    \\ & \leq \wlex ((x, 1), (x', 1)) \qquad \text{(because $\wlex$ is
         transitive)}
    \\ & = a(x, x') \qquad \text{(by definition of $\wlex$)}.
  \end{align*}
 
  Conversely, let $(x, y)$, $(x', y')$, and $(x'', y'')$ belong to $X
  \times Y$. If $y = y' = y'' = 1$, then
  \begin{align*}
    \wlex ((x, y), (x', y')) \otimes \wlex ((x', y'), (x'', y''))
    & = a(x, x') \otimes a(x', x'')  \qquad \text{(by definition of
      $\wlex$)}
    \\ & \leq a(x, x'') \qquad \text{(because $a$ is transitive)}
    \\ & = \wlex((x, y), (x'', y'')) \qquad \text{(by definition of $\wlex$)}.
  \end{align*}

  If, on the other hand, we have $y = y'' = 1$ but $y' \neq 1$, then
  \begin{align*}
    \wlex ((x, y), (x', y')) \otimes \wlex ((x', y'), (x'', y''))
    & = b(1, y') \otimes b(y', 1)  \qquad \text{(by definition of
      $\wlex$)}
    \\ & \leq a(x, x'') \qquad \text{(by hypothesis)}
    \\ & = \wlex((x, y), (x'', y'')) \qquad \text{(by definition of
         $\wlex$)}.
  \end{align*}

  Finally, suppose that we do not have $y = y'' = 1$. Then, we have
  \[\wlex((x, y), (x'',y'')) = b(y, y'')\]
  and the expression
  \begin{equation}
    \wlex ((x, y), (x', y')) \otimes \wlex ((x', y'), (x'',
    y''))\label{eq:9}
  \end{equation}
  is equal to one of the following:
  \begin{enumerate}[label = $(\roman*)$]
  \item\label{item:12} $a(x, x') \otimes b(y', y'')$,
  \item\label{item:13} $b(y, y') \otimes a(x', x'')$, or
  \item\label{item:14} $b(y, y') \otimes b(y', y'')$.
  \end{enumerate}
  Then, we have that~\eqref{eq:9} equals the expression
  in~\ref{item:12} if, and only if, $y = y' = 1$ and $y'' \neq
  1$. And, in that case,
  \begin{align*}
    \wlex ((x, y), (x', y')) \otimes \wlex ((x', y'), (x'', y''))
    & = a(x, x') \otimes b(y', y'')
    \\ & \leq b(1, y'') \qquad \text{(because $a(x, x') \leq k$)}
    \\ & = \wlex((x, y), (x'', y'')) \qquad \text{(because $y'' \neq 1$)}.
  \end{align*}
  Similarly, if~\eqref{eq:9} equals the expression in~\ref{item:13},
  then we must have $y \neq 1$ and $y' = y'' = 1$, and thus,
  \begin{align*}
    \wlex ((x, y), (x', y')) \otimes \wlex ((x', y'), (x'', y''))
    & = b(y, y') \otimes a(x', x'')
    \\ & \leq b(y, 1) \qquad \text{(because $a(x', x'') \leq k$)}
    \\ & = \wlex((x, y), (x'', y'')) \qquad \text{(because $y \neq 1$)}.
  \end{align*}
  Finally, if~\eqref{eq:9} equals the expression in~\ref{item:14}
  then,
  \begin{align*}
    \wlex ((x, y), (x', y')) \otimes \wlex ((x', y'), (x'', y''))
    & = b(y, y') \otimes b(y', y'')
    \\ & \leq b(y, y'') \qquad \text{(because $b$ is transitive)}
    \\ & = \wlex((x, y), (x'', y'')) 
    \\ & \hspace{2cm}\text{(because we do not
         have $y = y'' = 1$)}.\popQED \qed
  \end{align*}
\end{proof}

\begin{lemma}\label{l:4}
  The $\cV$-relation $\llex$ is transitive if, and only if, for every
  $x, x' \in X$ and every $y, y' \in Y$ with $y \neq y'$, we have
  $b(y, y') \otimes b(y', y) \leq a(x, x')$.
\end{lemma}
We omit the proof of this result as it is analogous to that of
Lemma~\ref{l:3} with the obvious adaptations.

While in the context of $\cV$-groups, the relation $\llex$ is the
biggest possible $\cV$-enrichment of a semidirect
product~\cite[Proposition~7.6]{ClementinoMontoli21}, in the setting of
$\cV$-monoids we have the following:
\begin{proposition}\label{p:1}
  Let $\alpha: Y \times X \to X$ be a monoid action and $c: (X \times
  Y) \times (X \times Y) \to V$ be a $\cV$-relation on $X \times Y$
  that turns $X \rtimes_\alpha Y$ into a $\cV$-monoid. Then, the
  following are equivalent:
  \begin{enumerate}
  \item\label{item:10} $(X, a) \overset{\iota_1} {\hookrightarrow} (X
    \rtimes_\alpha Y, c) \overset{\pi_2}{\underset{\iota_2}\rightleftarrows} (Y, b)$
    is a $U$-Schreier split extension of $\cV$-monoids,
  \item\label{item:11} $a \otimes b \leq c \leq \wlex$.
  \end{enumerate}
\end{proposition}
\begin{proof}
  Suppose that~\ref{item:10} holds.  Using~\ref{item:S2}
  and~\ref{item:S3}, and the fact that $(X \rtimes_\alpha Y, c)$ is a
  $\cV$-monoid, we have
  \begin{align*}
    a(x, x') \otimes b(y, y')
    & \leq c((x, 1), (x', 1)) \otimes c((0, y), (0, y'))
    \\ & \leq c((x, 1)(0, y), (x', 1)(0, y')) = c ((x, y),
         (x', y')),
  \end{align*}
  which proves that $a\otimes b\leq c$.  Now, using~\ref{item:S2}
  again, we have
  \[c((x, 1), (x',1)) \leq a(x, x') = \wlex((x, 1), (x', 1));\]
  and, by~\ref{item:S1},
  \[c((x, y), (x', y')) \leq b(y, y') = \wlex((x, y), (x', y')),\]
  where the last equality holds if at least one of $y$ and $y'$ is
  different from~$1$. This shows that $c \leq \wlex$.

  Conversely, let us suppose that $a\otimes b \leq c \leq \wlex$. By
  definition of $\wlex$, we have
  \[\wlex((x, y), (x', y')) \leq b(y, y')\]
  and thus~\ref{item:S1} holds because we are assuming that $c \leq
  \wlex$. This inequality also yields
  \[c((x, 1), (x',1)) \leq \wlex ((x, 1), (x', 1)) = a(x, x'),\]
  which is half of the equality~\ref{item:S2}. Finally, using the
  assumption $a\otimes b \leq c$, we have
  \[a(x, x') = (a\otimes b)((x, 1), (x', 1)) \leq c((x, 1), (x', 1))\]
  and
  \[b(y, y') = (a \otimes b)((0, y), (0, y')) \leq c((0, y), (0,
    y')),\]
  which show the other half of~\ref{item:S2} and~\ref{item:S3},
  respectively.
\end{proof}
As an immediate consequence, we have the following:
\begin{corollary}\label{c:1}
  Let $\alpha: Y \times X \to X$ be a monoid action and $c: (X \times
  Y) \times (X \times Y) \to V$ be a $\cV$-relation satisfying $a
  \otimes b\leq c \leq \wlex$. Then, the following are equivalent:
  \begin{enumerate}
  \item $(X, a) \overset{\iota_1} {\hookrightarrow} (X \rtimes_\alpha
    Y, c) \overset{\pi_2}{\underset{\iota_2}\rightleftarrows} (Y, b)$
    is a $U$-Schreier split extension of $\cV$-monoids,
  \item $(X \rtimes_\alpha Y, c)$ is a $\cV$-monoid.
  \end{enumerate}
\end{corollary}

We will now characterize under which conditions the pair $(X
\rtimes_\alpha Y, c)$ is a $\cV$-monoid, when $c$ is each of the
bounds identified in Proposition~\ref{p:1}, as well as when $c =
\llex$. 

Given a monoid action $\alpha: Y \times X \to X$, we consider the
function
\[\overline{\alpha}: Y \times X \to X \times Y, \qquad (y, x) \mapsto
  (\alpha(y, x), y).\]

The proof of the following result is included for the sake of
completeness, but we note that it is similar to that
of~\cite[Proposition~7.2]{ClementinoMontoli21}.

\begin{proposition}\label{p:3}
  Let $\alpha: Y \times X \to X$ be a monoid action. Then, the
  following are equivalent:
  \begin{enumerate}
  \item\label{item:1} $(X \rtimes_\alpha Y, a \otimes b)$ is a
    $\cV$-monoid,
  \item\label{item:2} $\overline{\alpha}$ is a $\cV$-functor.
  \end{enumerate}
\end{proposition}
\begin{proof}
  
  \ref{item:1} $\implies$ \ref{item:2}: Suppose that $(X
  \rtimes_\alpha Y, a \otimes b)$ is a $\cV$-monoid. Then, the
  following computations show that $\overline{\alpha}$ is a
  $\cV$-functor:
  \begin{align*}
    (b\otimes a)((y, x), (y', x'))
    &= a(0,0) \otimes  (b\otimes a)((y, x), (y', x')) \otimes b(1,1)
    \\ & = (a \otimes b) ((0, y), (0, y')) \otimes (a \otimes b)((x,
         1), (x',1))
    \\ & \leq (a\otimes b)(\overline{\alpha}(y, x),
         \overline{\alpha}(y', x')) \qquad \text{(because $(X
         \rtimes_\alpha Y, a \otimes b)$ is a $\cV$-monoid).}
  \end{align*}

  \ref{item:2} $\implies$ \ref{item:1}: We need to show that the
  operation on $X \rtimes_\alpha Y$ induces a $\cV$-functor, that is,
  that for all $x_1, x_1', x_2, x_2' \in X$ and $y_1, y_1', y_2, y_2'
  \in Y$, the following inequality holds:
  \begin{equation}
    (a\otimes b)((x_1, y_1), (x_1', y_1'))\otimes (a\otimes b)((x_2,
    y_2), (x_2', y_2')) \leq (a\otimes b)((x_1,y_1)(x_2,y_2), (x_1',
    y_1')(x_2',y_2')).\label{eq:1}
  \end{equation}
  Indeed, we have:
  \begin{align*}
    &\quad (a\otimes b)((x_1, y_1), (x_1', y_1'))\otimes (a\otimes b)((x_2,
      y_2), (x_2', y_2'))
    \\ & = a(x_1, x_1') \otimes (b \otimes a) ((y_1, x_2), (y_1', x_2'))
         \otimes b(y_2, y_2')
    \\ & \leq a(x_1, x_1') \otimes (a \otimes b)
         (\overline{\alpha}(y_1, x_2), \overline{\alpha}(y_1', x_2'))
         \otimes b(y_2, y_2') \qquad \text{(because
         $\overline{\alpha}$ is a $\cV$-functor)}
    \\ & \leq a(x_1 + \alpha(y_1, x_2), x_1' + \alpha(y_1', x_2'))
         \otimes b(y_1y_1', y_2y_2') \qquad \text{(because $(X, a)$
         and $(Y, b)$ are $\cV$-monoids)}
    \\ & = (a \otimes b)((x_1,y_1)(x_2,y_2), (x_1',
         y_1')(x_2',y_2')).\popQED \qed
  \end{align*}
\end{proof}

\begin{proposition}\label{p:2}
  Let $\alpha: Y \times X \to X$ be a monoid action. Then, the
  following are equivalent:
  \begin{enumerate}
  \item \label{item:4} $(X \rtimes_\alpha Y, \wlex)$ is a
    $\cV$-monoid,
  \item \label{item:5} the $\cV$-relation $\wlex$ is transitive and,
    for all $y_1, y_2, y_1', y_2' \in Y \setminus \{1\}$ satisfying
    $y_1y_2 = y_1'y_2' = y_1' y_2 = 1$, the following inequality
    holds:
    \begin{equation}
      b(y_1, y_1') \otimes b(y_2, y_2')  \leq \bigwedge_{x, x' \in X}
      a(x, x').\label{eq:13}
    \end{equation}
  \end{enumerate}
\end{proposition}
\begin{proof}
  \ref{item:4} $\implies$ \ref{item:5}: If $(X \rtimes_\alpha Y,
  \wlex)$ is a $\cV$-monoid then $\wlex$ is transitive. Let us show
  that~\eqref{eq:13} holds. Let $y_1, y_2, y_1', y_2' \in Y \setminus
  \{1\}$ be such that $y_1y_2 = y_1'y_2' = 1$ and $x, x' \in X$. Then,
  since $(X \rtimes_\alpha Y, \wlex)$ is a $\cV$-monoid, we have
  \begin{align*}
    b(y_1, y_1') \otimes b(y_2, y_2')
    & = \wlex ((x, y_1), (x', y_1')) \otimes \wlex ((0, y_2), (0,
      y_2'))
    \\ & \leq \wlex ((x, y_1y_2), (x', y_1'y_2')) = a(x, x').
  \end{align*}
  Since $x, x' \in X$ are arbitrary, we have~\eqref{eq:13}.

  \ref{item:5} $\implies$ \ref{item:4}: We have that $(X
  \rtimes_\alpha Y, \wlex)$ is a $\cV$-monoid if, and only if, the
  following inequality holds:
  \begin{equation}
    \wlex ((x_1, y_1), (x_1', y_1')) \otimes \wlex ((x_2, y_2),
    (x_2', y_2')) \leq \wlex((x_1 + \alpha(y_1, x_2), y_1y_2), (x_1'
    + \alpha(y_1', x_2'), y_1'y_2')).\label{eq:11}
  \end{equation}
  We consider the following cases, according to the value $v$ of the
  left-hand side of~\eqref{eq:11}:
  \begin{itemize}
  \item If $v = a(x_1, x_1') \otimes a(x_2, x_2')$ then, we must have
    $y_1 = y_1' = y_2 = y_2' = 1$ and inequality~\eqref{eq:11} follows
    from $(X, a)$ being a $\cV$-monoid.
  \item If $v = a(x_1, x_1') \otimes b(y_2, y_2')$ then it is because
    we have $y_1 = y_1' = 1$ but we do not have $y_2 = y_2' =
    1$. Thus, we cannot either have $y_1y_2 = y_1'y_2' = 1$ and thus,
    the right-hand side of~\eqref{eq:11} is $b(y_1y_2, y_1'y_2') =
    b(y_2, y_2')$ which is greater than or equal to $v$.
  \item If $v = b(y_1, y_1') \otimes a(x_2, x_2')$ then the argument
    is similar to the one of the previous case.
  \item If $v = b(y_1, y_1') \otimes b(y_2, y_2')$ then it is because
    neither $y_1 = y_1' = 1$ nor $y_2 = y_2' = 1$. We consider the
    following three further cases:
    \begin{itemize}
    \item If $y_1 y_2 \neq 1$ or $y_1'y_2' \neq 1$, then the
      right-hand side of~\eqref{eq:9} is $b(y_1y_2, y_1'y_2')$ which,
      since $(Y, b)$ is a $\cV$-monoid, is greater than or equal
      to~$v$.
    \item If $y_1y_2 = y_1'y_2' = 1$, but $y_1'y_2 \neq 1$, then
      \begin{align*}
        v & = b(y_1, y_1') \otimes b(y_2, y_2')
        \\ & \leq b(y_1y_2, y_1'y_2) \otimes b(y_1'y_2, y_1'y_2')
             \qquad \text{(because $(X, b)$ is a $\cV$-monoid)}
        \\ & = b(1, y_1'y_2) \otimes b(y_1'y_2, 1)
        \\ & \leq a(x_1 + \alpha(y_1, x_2), x_1' + \alpha(y_1',
             x_2')) \qquad \text{(by Lemma~\ref{l:3}).}
      \end{align*}
    \item If $y_1y_2 = y_1'y_2' = y_1'y_2 = 1$, then we
      use~\eqref{eq:13}.\popQED
    \end{itemize}
  \end{itemize}
\end{proof}
Given a function $\alpha: Y \times X \to X$ and $y \in Y$, we let
$\alpha_y: X \to X$ be defined by $\alpha_y (x) = \alpha(y, x)$.
\begin{proposition}\label{p:6}
  Let $\alpha: Y \times X \to X$ be a monoid action. Then, the
  following are equivalent:
  \begin{enumerate}
  \item\label{item:3} $(X \rtimes_\alpha Y, \llex)$ is a $\cV$-monoid,
  \item\label{item:6} $\llex$ is transitive, for every $y \in Y$,
    $\alpha_y$ is a $\cV$-functor and, for every $y_0, y, y' \in Y$
    with $y \neq y'$, if $y_0y = y_0 y'$ then
    \begin{equation}
      b(y, y') \leq \bigwedge_{x, x' \in X} a(\alpha(y_0, x),
      \alpha(y_0, x'))\label{eq:20}
    \end{equation}
    and if $yy_0 = y'y_0$ then
    \begin{equation}
      b(y, y') \leq \bigwedge_{x, x' \in X} a(x, x').\label{eq:24}
    \end{equation}
  \end{enumerate}
\end{proposition}
\begin{proof}
  \ref{item:3} $\implies$ \ref{item:6}: If $(X \rtimes_\alpha Y,
  \llex)$ is a $\cV$-monoid, then $\llex$ is, by definition,
  transitive. Let $y \in Y$. Then, for all $x, x' \in X$, we have
  \begin{align*}
    a(x, x')
    & = \llex ((0, y), (0, y)) \otimes \llex((x, 1), (x',1))
    \\ & \leq \llex((\alpha(y, x), y), (\alpha(y, x'), y)) \qquad
         \text{(because  $(X \rtimes_\alpha Y,
         \llex)$ is a $\cV$-monoid)}
    \\ & = a(\alpha_y(x), \alpha_y(x')),
  \end{align*}
  and therefore, $\alpha_y$ is a $\cV$-functor. Now, we let $y_0, y,
  y' \in Y$ be such that $y \neq y'$, and pick any $x, x' \in X$. If
  $y_0y = y_0y'$ then
  \begin{align*}
    b(y, y')
    & = \llex((0, y_0), (0, y_0)) \otimes \llex ((x, y),(x',y'))
    \\ & \leq \llex ((\alpha(y_0, x), y_0y), (\alpha(y_0, x'), y_0y')) \qquad
         \text{(because  $(X \rtimes_\alpha Y,
         \llex)$ is a $\cV$-monoid)}
    \\ & = a(\alpha(y_0, x), \alpha(y_0, x')) \qquad \text{(because  $y_0y = y_0y'$)}
  \end{align*}
  and this proves \eqref{eq:20}. If, on the other hand, we have $y y_0 =
  y'y_0$, then
  \begin{align*}
    b (y, y')
    & = \llex ((x, y), (x', y')) \otimes \llex((0, y_0), (0, y_0))
    \\ &\leq  \llex((x, yy_0), (x', y'y_0))\qquad
         \text{(because  $(X \rtimes_\alpha Y,
         \llex)$ is a $\cV$-monoid)}
    \\ & = a(x, x') \qquad \text{(because  $y y_0 = y'y_0$)}
  \end{align*}
  which shows~\eqref{eq:24}.

  \ref{item:6} $\implies$ \ref{item:3}: Since $\llex$ is always
  reflexive, we have that $(X \rtimes_\alpha Y, \llex)$ is a
  $\cV$-category. Thus, it suffices to show that, for all $x_1, x_1',
  x_2, x_2' \in X$ and $y_1, y_1', y_2, y_2' \in Y$, the following
  inequality holds:
  \begin{equation}
    \label{eq:10}
     \llex ((x_1, y_1), (x_1', y_1')) \otimes \llex((x_2, y_2), (x_2',
    y_2')) \leq \llex ((x_1 + \alpha(y_1, x_2), y_1y_2), (x_1' +
    \alpha(y_1', x_2'), y_1'y_2')).
  \end{equation}
  For that, we consider the following cases:
  \begin{itemize}
  \item If $y_1 = y_1'$ and $y_2 = y_2'$, then \eqref{eq:10} is
    equivalent to
    \[a(x_1, x_1') \otimes a(x_2, x_2') \leq a(x_1 + \alpha(y_1, x_2),
      x'_1 + \alpha(y'_1, x'_2)).\]
    Using the fact that $\alpha_{y_1} = \alpha_{y_1'}$ is a
    $\cV$-functor and that $(X, a)$ is a $\cV$-monoid, we have
    \[a(x_1, x_1') \otimes a(x_2, x_2') \leq a(x_1, x_1') \otimes
      a(\alpha_{y_1}(x_2), \alpha_{y'_1}( x'_2)) \leq a(x_1 +
      \alpha(y_1, x_2), x'_1 + \alpha(y'_1, x'_2)),\] as required.
  \item If $y_1 = y_1'$, $y_2 \neq y_2'$ and $y_1y_2 = y_1'y_2'$ then
    \eqref{eq:10} is equivalent to
    \[a(x_1, x_1') \otimes b(y_2, y_2') \leq a(x_1 + \alpha(y_1, x_2),
      x'_1 + \alpha(y'_1, x'_2)),\]
    which holds because, by \eqref{eq:20}, the inequality
    \[a(x_1, x_1') \otimes b(y_2, y_2') \leq a(x_1, x_1') \otimes
      a(\alpha(y_1, x), \alpha(y_1', x'))\]
    holds and $(X \rtimes_\alpha, \llex)$ is a $\cV$-monoid.
  \item If $y_1 \neq y_1'$, $y_2 = y_2'$ and $y_1y_2 = y_1'y_2'$ then
    \eqref{eq:10} is equivalent to
    \[b(y_1, y_1') \otimes a(x_2, x_2') \leq a(x_1 + \alpha(y_1, x_2),
      x'_1 + \alpha(y'_1, x'_2)),\]
    which holds thanks to~\eqref{eq:24}.

  \item If $y_1y_2 \neq y_1'y_2'$ then at least one of the equalities
    $y_1 = y_1'$ and $y_2 = y_2'$ must fail. If $y_i \neq y_i'$ for
    exactly one $i \in \{1, 2\}$, then we have
    \begin{align*}
      \llex ((x_1, y_1), (x_1', y_1')) \otimes \llex((x_2, y_2), (x_2',
      y_2'))  & \leq b(y_i, y_i')
      \\ & \leq b(y_1y_2, y_1'y_2') \qquad \text{(by Lemma~\ref{l:8})}
      \\ & =  \llex ((x_1 + \alpha(y_1, x_2), y_1y_2), (x_1' +
           \alpha(y_1', x_2'), y_1'y_2')).
    \end{align*}
    If we have both $y_1 \neq y_1'$ and $y_2 \neq y_2'$
    then~\eqref{eq:10} is equivalent to
    \[b(y_1, y_1') \otimes b(y_2, y_2') \leq b(y_1y_2, y_1'y_2'),\]
    which holds because $(Y, b)$ is a $\cV$-monoid.\popQED
    \end{itemize}
\end{proof}
We have thus provided characterizations of the $\cV$-monoids $(X
\rtimes_\alpha Y, c)$ for three particular instances of~$c$. Now,
since the underlying diagram of monoids of a $U$-Schreier split extension
of $\cV$-monoids forms a Schreier split extension of monoids, in the
diagram~\eqref{eq:3}, there is a unique set map $q:Z \to X$
satisfying~\eqref{eq:2}. We recall that, in~\eqref{eq:8}, such a map
is the first projection $\pi_1: X \rtimes_\alpha Y \to X$. It is then
natural to ask for necessary and sufficient conditions for having that
this map is also a $\cV$-functor. We provide such in the next result.
\begin{lemma}
  If $(X, a) \overset{\iota_1} {\hookrightarrow} (X \rtimes_\alpha Y,
  c) \overset{\pi_2}{\underset{\iota_2}\rightleftarrows} (Y, b)$ is a
  $U$-Schreier split extension of $\cV$-monoids then, $\pi_1$ is a
  $\cV$-functor if, and only if, $c \leq a \wedge b$.
\end{lemma}
\begin{proof}
  By definition, $\pi_1$ is a $\cV$-functor if, and only if, for all
  $x, x' \in X$ and $y, y' \in Y$, the following inequality holds:
  \[c((x, y), (x', y')) \leq a(x, x').\]
  Noting that, by definition of $\wlex$, we have $\wlex ((x, y), (x',
  y'))\leq b(y, y')$ and, by Proposition~\ref{p:1}, we have $c \leq
  \wlex$, the forward implication follows. The backwards implication
  is trivial.
\end{proof}

In the remaining of the section, we will focus on the case where $Y$
is a group. When that is the case, Propositions~\ref{p:2}
and~\ref{p:6} may be considerably simplified as follows.

\begin{corollary}[of Proposition~\ref{p:2}]\label{c:3}
  Let $\alpha: Y \times X \to X$ be a monoid action and suppose that
  $Y$ is a group. Then, the following are equivalent:
  \begin{enumerate}
  \item\label{item:21} $(X \rtimes_\alpha Y, \wlex)$ is a $\cV$-monoid,
  \item\label{item:22} $\llex = \wlex$.
  \end{enumerate}
\end{corollary}
\begin{proof}
  We recall that, by Lemma~\ref{l:1}, $\llex = \wlex$ if, and only if,
  either $Y$ is trivial or $a(x, x') = k$ for all $x, x' \in X$. Now,
  if $(X \rtimes_\alpha Y, \wlex)$ is a $\cV$-monoid and $Y$ is
  non-trivial then we may pick $y \in Y \setminus \{1\}$ and, by
  Proposition~\ref{p:2}, \eqref{eq:13} holds for $y_1 = y_1' = y$ and
  $y_2 = y_2' = y^{-1}$, which yields $a(x, x') = k$ for all $x, x'
  \in X$. This shows that \ref{item:21} implies
  \ref{item:22}. Conversely, if $\llex = \wlex$ then either $Y$ is
  trivial and $\wlex = a$, or $a(x, x') = k$ for all $x, x' \in X$. In
  either case, we have that $\wlex$ is transitive. Furthermore, it is
  clear that \eqref{eq:13} holds. Thus, by Proposition~\ref{p:2}, we
  may conclude that $(X \rtimes_\alpha Y, \wlex)$ is a $\cV$-monoid,
  as required.
\end{proof}
\begin{corollary}[of Proposition~\ref{p:6}]\label{c:2}
  Let $\alpha: Y \times X \to X$ be a monoid action and suppose that
  $Y$ is a group. Then, the following are equivalent:
  \begin{enumerate}
  \item $(X \rtimes_\alpha Y, \llex)$ is a $\cV$-monoid,
  \item $\llex$ is transitive and $\alpha_y$ is a $\cV$-functor for
    all $y \in Y$.
  \end{enumerate}
\end{corollary}
\begin{proof}
  This is a trivial consequence of Proposition~\ref{p:6} as, when $Y$
  is a group, there are no $y_0, y, y' \in Y$ for which $y\neq y'$ and
  $y_0y = y_0y'$ or $yy_0 = y'y_0$.
\end{proof}
We have already identified necessary and sufficient conditions for
having that $\llex$ and $\wlex$ are transitive $\cV$-relations. Along
the same lines, we may also show that, in the setting of
$\cV$-groups, the condition identified
in~\cite[Theorem~7.4]{ClementinoMontoli21} is necessary and sufficient
for $\llex$ being transitive. More precisely, if $X$ and $Y$ are
groups, then $\llex$ is transitive if, and only if, the inequality
\[b(y, 1) \otimes b(1, y) \leq a(x, 0)\]
holds for every $x \in X$ and $y \in Y \setminus \{1\}$.  Thus,
Corollary~\ref{c:2} is a generalization
of~\cite[Theorem~7.4]{ClementinoMontoli21}. On the other hand,
Corollaries~\ref{c:3} and~\ref{c:2} also imply that, in the setting of
$\cV$-groups, if $(X \rtimes_\alpha Y, \wlex)$ is a $\cV$-monoid, then
the relations $\llex$ and $\wlex$ coincide. That is no surprise as
in~\cite[Proposition~7.6]{ClementinoMontoli21} it is shown that
$\llex$ is an upper bound of all relations $c$ turning $(X
\rtimes_\alpha Y, c)$ into a $\cV$-group. The next example shows that,
unlike what happens for $\cV$-groups, considering the relation $\wlex$
is not redundant for $\cV$-monoids in general.

\begin{example}
  Let $\cV = \two$, so that $\vmon$ is the category of preordered
  monoids. We consider the following preordered monoids:
  \begin{itemize}
  \item $\NN$ is the monoid of natural numbers equipped with the usual
    order relation,
  \item $\dot \NN$ is the monoid of natural numbers equipped with the
    preorder $\dot \leq$ defined by
    \[0 \,\dot\leq\, n, \text{ for all $n \in \NN$,}\qquad \text{and}
      \qquad n \,\dot \leq\, m, \text{ for all $n, m \in \NN \setminus
        \{0\}$}.\]
  \end{itemize}
  It is easy to verify that both $\NN$ and $\dot \NN$ are indeed
  preordered monoids and, using Lemma~\ref{l:3} we may also check that
  $\wlex$ is transitive and thus, $(\NN \times \dot \NN, \leq_\wlex)$
  is a $\two$-category (or preordered set). Moreover, the
  lexicographic and weak lexicographic relations do not coincide in
  this case: for instance, $(2,2)$ is below $(1,2)$ is the weak
  lexicographic relation, but not in the lexicographic one. Finally,
  we check that $(\NN \times \dot \NN, \leq_\wlex)$ is a preordered
  monoid. First observe that
  \[(m, n) \leq_\wlex (m', n') \iff (n = n' = 0 \text{ and } m \leq
    m') \text{ or } (n' \neq 0).\]
  It is then clear that $\leq_\wlex$ is invariant by shifting, as
  required.
\end{example}

\section{Group-like behaved quantale enriched
  monoids}\label{sec:vmon-star}

A crucial and useful property in the study of preordered groups is the
fact that the preorder relation of a preordered group is completely
determined by its cone of positive elements, in the following sense:
If $(G, \leq, +)$ is a preodered group and $P_G = \{x \in G \mid x
\geq 0\}$ is the cone of positive elements of $G$, then $x \leq y$ if,
and only if, $y \in P_G + x$.  That is no longer the case for
preordered monoids as witnessed by
\cite[Example~1]{Martins-FerreiraSobral21}. Indeed, if $(M, +, \leq)$
is a preordered monoid, $P_M = \{x \in M \mid x \geq 0\}$ is its cone
of positive elements, and $\leq_{P_M}$ is the preorder on~$M$ defined
by
\begin{equation}
  x \leq_{P_M} y \iff y \in P_M+ x,\label{eq:6}
\end{equation}
then $(M, +, \leq_{P_M})$ is a preordered monoid if, and only if,
$P_M$ is a right normal submonoid of~$M$
\cite[Proposition~2]{Martins-FerreiraSobral21} and this condition does
not even guarantee that $\leq_{P_M}$ is the preorder on~$M$
\cite[page~5]{Martins-FerreiraSobral21}. Given the importance, in the
context of preordered groups, of having that the relations
$\leq_{P_M}$ and $\leq$ coincide, in~\cite{Martins-FerreiraSobral21}
the authors restrict their study of split extensions to those
preordered monoids for which that is the case.

In this section we will start by investigating which property can play
the role of right normality in the more general context of
$\cV$-monoids and, in the spirit of~\cite{Martins-FerreiraSobral21},
we will restrict our study to the subclass of $\cV$-monoids that, in a
sense that we will make precise soon, behave like $\cV$-groups.

Let $(M, \leq, +)$ be a preordered monoid and let $P_M = \{x \in M
\mid x \geq 0\}$ be its cone of positive elements. Then, seeing $M$ as
a $\two$-monoid $(M, a, +)$, for \[a(x, y) =
\begin{cases}
  1, \text{ if $x \leq y$,} \\ 0, \text{ otherwise},
\end{cases}\]
we have that $P_M$ is the preimage of $1 = \top$ under the projection
$a(0, \_): M \to \two$. This very simple observation opens the door to
a generalization of the results of~\cite{Martins-FerreiraSobral21} to
the setting of $\cV$-monoids. Indeed, we will now focus on those
$\cV$-monoids $(X, a)$ whose $\cV$-relation~$a$ is determined by the
projection $a(0, \_): X \to V$.

Given a $\cV$-monoid $(X, a)$, we will denote by $P_a$ the map $P_a :
X \to V$ defined by
\begin{equation}
  P_a(x) = a(0, x),\label{eq:23}
\end{equation}
for all $x \in X$. We note that, in the case where $(X, a)$ is a
$\cV$-group, the $\cV$-relation $a$ is completely determined by its
projection $P_a$. Indeed, by Lemma~\ref{l:8}, we have
\[a(x, y) \leq a (0, y-x) \leq a(x, y)\] and, therefore, the equality
\[a(x, y) = P_a(y - x)\]
holds. More generally, we will consider those $\cV$-monoids $(X, a)$
satisfying
\begin{equation}
  a(x, y) = \bigvee \{P_a(w) \mid y = w + x\},\label{eq:16}
\end{equation}
for all $x, y \in X$.  We observe that, in the case where $X$ is a
group, the right-hand side of~\eqref{eq:16} is simply $P_a(y-x)$.  We
further observe that, for a $\cV$-category $(X, a)$, the function $P =
P_a$ satisfies the following two properties:
\begin{enumerate}[label = (M.\arabic*)]
\item\label{item:M1} $k \leq P(0)$,
\item\label{item:M2} $P(x) \otimes P(y) \leq P(x + y)$, for all $x, y
  \in X$.
\end{enumerate}
These two properties turn out to be crucial when defining a
$\cV$-monoid structure on a given monoid~$X$ out of a function $X \to
V$. Indeed, such a function $P: X \to V$, we consider the
$\cV$-relation $a_P: X \times X \to V$ defined by
\begin{equation}
  a_P(x, y) = \bigvee \{P(w) \mid y = w + x\}.\label{eq:21}
\end{equation}

\begin{proposition}\label{p:4}
  Let $X$ be a monoid, and let $P: X \to V$ be a function. Then,
  \begin{enumerate}
  \item\label{item:18} $a_P$ is reflexive if, and only if, $P$
    satisfies \ref{item:M1};
  \item\label{item:19} $a_P$ is transitive if, and only if, $P$
    satisfies \ref{item:M2};
  \item\label{item:20} if $(X, a_P)$ is a $\cV$-category, then the
    monoid operation on $X$ is a $\cV$-functor if, and only if, for
    all $x, z \in X$, $P$ satisfies
    \begin{equation}
    P(x) \leq \bigvee \{P(w) \mid z + x = w +
    z\}.\tag*{(M.3)}\label{item:M3}
  \end{equation}
  \end{enumerate}
  In particular, $(X, a_P)$ is a $\cV$-monoid if, and only if, $P$
  satisfies properties \ref{item:M1}, \ref{item:M2}, and
  \ref{item:M3}.
\end{proposition}
\begin{proof}
  Noticing that $P_{a_P} = P$, if $a_P$ is reflexive and transitive
  then $P$ satisfies \ref{item:M1} and \ref{item:M2},
  respectively. Thus, the forward implications of \ref{item:18} and
  \ref{item:19} hold. Suppose that \ref{item:M1} holds. Then, $a_P$ is
  reflexive because
  \[a_P (x, x) = \bigvee \{P(w) \mid x = w + x \} \geq P(0)
    \overset{\text{\ref{item:M1}}}\geq k.\]
  If \ref{item:M2} holds, then $a_P$ is transitive because
  \begin{align*}
    a_P(x, y) \otimes a_P(y, z)
    & = (\bigvee \{P(w) \mid y = w + x\}) \otimes (\bigvee \{P(w) \mid
      z = w + y\})
    \\ & = \bigvee \{P(w) \otimes P(w') \mid  y = w + x\text{ and }z
         =  w' + y \}
    \\ & = \bigvee \{P(w') \otimes P(w) \mid  y = w + x\text{ and }z
         =  w' + y \}
    \\ & \overset{\text{\ref{item:M2}}}\leq \bigvee \{P(w' + w) \mid  y = w + x \text{ and }z
         = w' + y\}
    \\ & \leq \bigvee \{P(w' + w) \mid  z = w' + w + x\} = a_P(x,
         z),
  \end{align*}
  and this finishes the proofs of \ref{item:18} and \ref{item:19}.  
    
  Let us now prove~\ref{item:20}. Suppose that $(X, a_P)$ is a
  $\cV$-category. If the monoid operation on~$X$ is a $\cV$-functor
  then, using the second inequality stated in Lemma~\ref{l:8}, for all
  $x, z \in X$, we have
  \[P(x) = a_P(0, x) \leq a_P(z, z + x) = \bigvee\{P(w) \mid z + x = w
    + z\}.\]
  Conversely, let us verify that the inequalities~\eqref{eq:4} of
  Lemma~\ref{l:8} hold. First note that we always have $a_P(x, y) \leq
  a_P(x + z, y + z)$. Indeed,
  \[a_P(x, y) = \bigvee \{P(w) \mid y = w + x\} \leq \bigvee \{P(w)
    \mid y + z = w + x + z\} = a_P(x+z, y+z).\]
  Using~\ref{item:M3}, we may deduce that
  \begin{align*}
    a_P(x, y)
    & = \bigvee \{P(w) \mid y = w + x\}
    \\ & \leq \bigvee \{\bigvee \{P(w')\mid z + w = w' + z\} \mid y =
         w + x\}
    \\ & =  \bigvee \{P(w')\mid z + w = w' + z \text{ and } y =
         w + x\}
    \\ & \leq \bigvee \{P(w')\mid  z + y = w' + z + x\} = a_P(z + x, z
         + y). \popQED\qed
  \end{align*}
\end{proof}

We remark that, in the case of preordered monoids, having $P(x) \leq
\bigvee \{P(w) \mid z + x = w + z\}$ for all $x, z \in X$ means that,
if $0 \leq x$ and $z \in X$, then there exists $w$ such that $0 \leq
w$ and $z + x = w + z$. In other words, this is to say that $z + P
\subseteq P + z$, that is, $P$ is right normal.

We will denote by $\vmon^*$ the full subcategory of $\vmon$ determined
by the $\cV$-monoids $(X, a)$ satisfying $a = a_{P_a}$
(recall~\eqref{eq:23} and~\eqref{eq:21}). In particular, by
Proposition~\ref{p:4}, when that is the case, $P_a$ must
satisfy \ref{item:M1}--\ref{item:M3}.

The remaining of this section will be devoted to the characterization
of the $U$-Schreier split extensions of $\cV$-monoids of the form
\[(X, a) \overset{k} {\hookrightarrow} (Z, c)
  \overset{p}{\underset{s}\rightleftarrows} (Y, b),\]
where $(X, a)$, $(Y, b)$, and $(Z, c)$ belong to $\vmon^*$.
\begin{definition}\label{sec:V-action}
  Let $(X, a)$ and $(Y, b)$ be objects of $\vmon^*$. A
  \emph{$\cV$-enriched action} of $(Y, b)$ on $(X, a)$ is a pair
  $(\alpha, P)$, where $\alpha: Y \times X \to X$ is a monoid action
  and $P: X \times Y \to V$ is a function satisfying the following
  axioms:
  \begin{enumerate}[label = (E.\arabic*)]\setcounter{enumi}{-1}
  \item\label{item:E0} $P(x, y) \leq P_b(y)$ for all $(x, y) \in X \times Y$,
  \item\label{item:E1} $P_b(y) \leq P(0, y)$ for all $y \in Y$,
  \item\label{item:E2} $P_a(x) = P(x, 1)$ for all $x \in X$,
  \item\label{item:E3} $P(x, y) \otimes P(x', y') \leq P(x + \alpha(y,
    x'), yy')$, for all $(x, y), (x', y') \in X \times Y$,
  \item\label{item:E4} $P(x, y) \leq \bigvee \{P(x', y') \mid x_0 +
    \alpha(y_0, x) = x' + \alpha(y', x_0) \text{ and } y_0 y =
    y'y_0\}$, for all $(x, y), (x_0, y_0) \in X \times Y$.
  \end{enumerate}
\end{definition}

We note that properties~\ref{item:E3} and~\ref{item:E4} are nothing
but properties~\ref{item:M2} and \ref{item:M3}, respectively, stated
for the function $P: X \times Y \to V$ and for the monoid $X
\rtimes_\alpha Y$.

\begin{theorem}\label{t:2}
  Let $(X, a)$ and $(Y, b)$ be objects of $\vmon^*$. Then, up to
  isomorphism, there is a one-to-one correspondence between $U$-Schreier
  split extensions $(X, a) \overset{k} {\hookrightarrow}(Z, c)
  \overset{p}{\underset{s}\rightleftarrows} (Y, b)$ of $\cV$-monoids,
  with $(Z, c)$ lying in $\vmon^*$, and $\cV$-enriched actions of $(Y,
  b)$ on $(X, a)$.
\end{theorem}
\begin{proof}
  Let $(X, a) \overset{k} {\hookrightarrow}(Z, c)
  \overset{p}{\underset{s}\rightleftarrows} (Y, b)$ be a $U$-Schreier split
  extension of $\cV$-monoids in the category $\vmon^*$, let $s$ be the
  section of~$p$, and let $q: Z \to X$ be the unique set map
  satisfying~\eqref{eq:2}. We consider the monoid action $\alpha$
  defined by $\alpha(y, x) = q(s(y) \star k(x))$. By
  Theorem~\ref{t:1}, we know that $\varphi: X \rtimes_\alpha Y \to Z$
  and $\psi: Z \to X \rtimes_\alpha Y$ defined by $\varphi(x, y) =
  k(x) \star s(y)$ and by $\psi(z) = (q(z), p(z))$, respectively, are
  mutually inverse monoid isomorphisms. We let $P: X \times Y \to V$
  be defined by
  \[P(x, y) = P_c(k(x) \star s(y)) = P_c(\varphi(x, y))\]
  and we claim that $(\alpha, P)$ is a $\cV$-enriched action of $(Y,
  b)$ on $(X, a)$.
  
  \ref{item:E0}: Since $p$ is a $\cV$-functor, we have
  \[P(x, y) = P_c(k(x)\star s(y)) \leq P_b(p(k(x) \star s(y))) =
    P_b(y),\]
  where the last equality uses that $\psi \circ \varphi$ is the
  identity map on~$X\times Y$.
  
  \ref{item:E1}: Since $s$ is a $\cV$-functor, we have
  \[P_b(y) \leq P_c(s(y)) = P(0, y).\]
  
  \ref{item:E2}: Since $k$ is the kernel of~$p$, we have
  \[P_a(x) = P_c(k(x)) = P(x, 1).\]

  For proving \ref{item:E3} and \ref{item:E4}, we use the fact that
  $(Z, c)$ belongs to $\vmon^*$ and thus, the function $P_c$ satisfies
  \ref{item:M2} and \ref{item:M3}.

  \ref{item:E3}: By~\ref{item:M2}, we have
  \begin{align*}
    P(x, y) \otimes P(x', y')
    & = P_c(k(x)\star s(y)) \otimes P_c(k(x')\star s(y')) \leq
      P_c(k(x)\star s(y) \star k(x')\star s(y'))
    \\ & =P_c(\varphi(x, y) \star \varphi(x', y')) = P_c(\varphi(x +
         \alpha(y, x'), yy'))
    \\ & = P(x + \alpha(y, x'), yy').
  \end{align*}

  \ref{item:E4}: By~\ref{item:M3}, for all $z, z_0 \in Z$, we have
  \[P_c(z) \leq \bigvee \{P_c(z') \mid z_0 \star z = z' \star z_0\}.\]
  Since $\varphi$ is a monoid isomorphism, this is equivalent to
  having
  \[P_c(\varphi(x, y)) \leq \bigvee \{P_c(\varphi(x', y')) \mid
    \varphi(x_0 + \alpha (y_0, x), y_0y) = \varphi(x' + \alpha(y',
    x_0), y'y_0)\}\]
  for all $(x, y), (x_0, y_0) \in X \times Y$, and having
  \[\varphi(x_0 + \alpha (y_0, x), y_0y) = \varphi(x' + \alpha(y',
    x_0), y'y_0)\]
  is equivalent to having
  \[x_0 + \alpha(y_0, x) = x' + \alpha(y', x_0) \quad\text{ and }\quad
    y_0 y = y'y_0.\] Thus, we have \ref{item:E4}.

  Conversely, let $(\alpha, P)$ be a $\cV$-enriched action of $(Y, b)$
  on $(X, a)$. By Proposition~\ref{p:4}, taking
  \[ a_P((x, y), (x', y')) = \bigvee \{P(x'', y'') \mid (x', y') =
    (x'', y'') \star (x, y)\},\]
  yields a $\cV$-monoid $(X \rtimes_\alpha Y, a_P)$ provided $P$
  satisfies properties \ref{item:M1}, \ref{item:M2}, and
  \ref{item:M3}. As already observed, \ref{item:M2} and \ref{item:M3}
  hold because so do \ref{item:E3} and \ref{item:E4}, respectively. To
  show \ref{item:M1}, we may use \ref{item:E2} and the fact that $P_a$
  satisfies \ref{item:M1}. Now, by Corollary~\ref{c:1}, to conclude
  that
  \[(X, a) \overset{\iota_1} {\hookrightarrow} (X \rtimes_\alpha Y,
    a_P) \overset{\pi_2}{\underset{\iota_2}\rightleftarrows} (Y, b)\]
  is a $U$-Schreier split extension of $\cV$-monoids, it suffices to show
  that $a \otimes b \leq a_P \leq \wlex$. Since $(X, a)$ and $(Y, b)$
  belong to $\vmon^*$, for all $x, x' \in X$ and $y, y' \in Y$, we
  have
  \begin{equation}
    a(x, x') = \bigvee \{P_a(x'') \mid x' = x'' + x\}
    \overset{\ref{item:E2}} = \bigvee \{P(x'', 1) \mid x' = x'' + x\}
    = a_P((x, 1), (x',1))\label{eq:22}
  \end{equation}
  and
  \[b(y, y') = \bigvee \{P_b(y'') \mid y' = y'' y\}
    \overset{\ref{item:E1}}\leq \bigvee \{P(0, y'') \mid y' = y'' y\}
    = a_P((0, y), (0, y')).\]
  Thus,
  \[a(x, x') \otimes b(y, y') \leq a_P((x, 1), (x',1)) \otimes a_P((0,
    y), (0, y')) \leq a_P((x,y) , (x', y')),\]
  where the last inequality holds because $(X \rtimes_\alpha Y, a_P)$
  is a $\cV$-monoid. This shows that $a \otimes b \leq a_P$. Now,
  by~\eqref{eq:22}, $a_P((x, 1), (x', 1)) = a(x, x')$ and, for
  arbitrary $y, y'$,
  \[a_P((x, y), (x', y')) = \bigvee \{P(x'', y'') \mid (x', y') =
    (x'', y'') \star (x, y)\} \overset{\ref{item:E0}} \leq \bigvee
    \{P_b(y'') \mid y' = y'' y\} = b(y, y'),\]
  which finishes showing that $a_P \leq \wlex$.
  
  It remains to check that the two correspondences just described are
  mutually inverse. It is easily seen that $P = P_{a_P}$. Thus, taking
  Theorem~\ref{t:1} into account, it remains to show that, for a
  $U$-Schreier split extension $(X, a) \overset{k} {\hookrightarrow}(Z, c)
  \overset{p}{\underset{s}\rightleftarrows} (Y, b)$, if $(\alpha, P)$
  is the corresponding $\cV$-enriched action, then the monoid
  isomorphisms $\varphi$ and $\psi$ define morphisms of $\cV$-monoids
  when $X \rtimes_\alpha Y$ is equipped with the $\cV$-relation
  $a_P$. Indeed, we have
  \begin{align*}
    a_P((x, y), (x',y'))
    & = \bigvee \{P(x'', y'') \mid (x', y') = (x'', y'') \star (x,
      y)\}
    \\ & = \bigvee \{P_c(\varphi(x'', y'')) \mid (x', y') = (x'', y'')
         \star (x, y)\}\quad \text{(by definition of $P$)}
    \\ & =  \bigvee \{P_c(\varphi(x'', y'')) \mid \varphi(x', y') =
         \varphi(x'', y'')  \star \varphi(x, y)\}\quad \text{(because
         $\varphi$ is a}
    \\ & \hspace{100mm}\text{monoid isomorphism)}
    \\ & = c(\varphi(x, y), \varphi(x', y'))\quad \text{(because $(Z,
         c)$ belongs to $\vmon^*$)}\popQED \qed
  \end{align*}
\end{proof}
We finish this section by noting that, as
in~\cite{Martins-FerreiraSobral21}, we could also have defined
morphisms of $U$-Schreier split extensions of $\cV$-monoids and of
$\cV$-enriched actions in the obvious way, thereby forming two
categories that could be proved to be equivalent by following the
ideas used in the proof of Theorem~\ref{t:2}. We do not include
details on this as we believe no further meaningful mathematical
knowledge of the structures involved would be added.

\section{The case of preordered monoids}\label{sec:preorderd-mon}

In this section, we analyze the results of Section~\ref{sec:vmon-star}
in the case of preordered monoids and compare them with those
of~\cite{Martins-FerreiraSobral21}.

Let $\cV = \two$, so that $\vmon$ can be identified with the category
of preordered monoids, and let $(X, a)$ be a $\two$-monoid. The
ensuing preorder on $X$ will be denoted by~$\leq_X$. A function $P : X
\to \two$ is uniquely determined by the subset $P^{-1}(\{\top\})
\subseteq X$ and, conversely, each subset $P \subseteq X$ uniquely
determines a function $X \to \two$. We will often abuse notation and
identify a subset $P \subseteq X$ with the function $P:X \to \two$ it
defines. As already mentioned, under this identification, $P_a$ as
defined in the previous section is the cone~$P_X$ of positive elements
of~$X$. Now, given a subset $P \subseteq X$, the $\two$-relation $a_P$
defined in~\eqref{eq:21} induces the preorder~$\leq_P$ on~$X$ given by
\[x \leq_P y \iff y \in P + x,\]
for all $x, y \in P$. Indeed, we have $x \leq_P y$ if, and only if,
$a_P(x, y) = \top$, which holds if, and only if, there exists some $w
\in X$ satisfying $P(w) = \top$ and $y = w + x$. We further observe
that $P \subseteq X$ satisfies properties \ref{item:M1} and
\ref{item:M2} if, and only if, $0 \in P$ and $x + y \in P$ whenever
$x, y \in P$, respectively. That is, if, and only if, $P$ is a
submonoid of~$X$. In turn, as it was already explained, requiring
\ref{item:M3} is equivalent to requiring that $P$ is right
normal. Thus, our Proposition~\ref{p:4} is a generalization
of~\cite[Proposition~2]{Martins-FerreiraSobral21} to the setting of
$\cV$-monoids. Moreover, the category ${\bf OrdMon}^*$ studied
in~\cite{Martins-FerreiraSobral21} is our category $\twomon^*$.  Let
us state Definition~\ref{sec:V-action} for $\cV = \two$.
\begin{proposition}\label{p:5}
  Let $X$ and $Y$ be objects of $\twomon^*$. A \emph{$\two$-enriched
    action} of $Y$ on $X$ is a pair $(\alpha, P)$, where $\alpha: Y
  \times X \to X$ is a monoid action and $P \subseteq X \times P_Y$
  satisfies the following axioms:
  \begin{enumerate}[label = (B.\arabic*)]\setcounter{enumi}{-1}
  \item\label{item:B0} $P \cap (X \times \{1\}) \subseteq P_X \times \{1\}$,
  \item\label{item:B1} $\{0\} \times P_Y \subseteq P$,
  \item\label{item:B2} $P_X \times \{1\} \subseteq P$,
  \item\label{item:B3} if $(x, y)$ and $(x', y')$ belong to $P$ then
    so does $(x + \alpha(y, x'), yy')$,
  \item\label{item:B4} if $(x, y) \in P$ and $(x_0, y_0) \in X \times
    Y$, then there exists $(x', y') \in P$ such that $x_0 +
    \alpha(y_0, x) = x' + \alpha(y', x_0)$ and $y_0 y = y'y_0$.
  \end{enumerate}
\end{proposition}
\begin{proof}
  This is a straightforward translation of
  Definition~\ref{sec:V-action}, with \ref{item:E0} corresponding to
  the condition $P \subseteq X \times P_Y$, \ref{item:B0} and
  \ref{item:B2} corresponding to \ref{item:E1} and, for $i = 1, 3, 4$,
  (B.$i$) corresponding to (E.$i$).
\end{proof}

Let us now recall the main result
of~\cite{Martins-FerreiraSobral21}. \emph{In loc. cit.}, preordered
actions are defined as follows:
\begin{definition}
  [{\cite[Definition
    4]{Martins-FerreiraSobral21}}]\label{sec:preaction} Let $X$ and
  $Y$ belong to $\twomon^*$. A \emph{preordered action} of $Y$ on
  $X$ is a pair $(\alpha, \xi)$, where $\alpha: Y \times X \to X$ is a
  monoid action and $\xi: X \times P_Y \to X$ is a function satisfying
  \begin{enumerate}[label = (A.\arabic*)]
  \item\label{item:A1} $\xi(0, y) = 0$, for all $y \in P_Y$,
  \item\label{item:A2} $\xi(x, 1) = x$, for all $x \in P_X$,
  \item\label{item:A3} if $\xi(x, y)= x$ and $\xi(x', y')= x'$, then
    $\xi(x + \alpha(y, x'), yy') = x + \alpha(y, x')$,
  \item\label{item:A4} if $(x, y) \in X \times P_Y$ and $(x_0, y_0)
    \in X \times Y$, then there exists $(x', y') \in X \times P_Y$
    such that $x_0 + \alpha(y_0, x) = x' + \alpha(y', x_0)$, $\xi(x',
    y') = x'$, and $y_0 y = y'y_0$.
  \end{enumerate}
\end{definition}

The intuitive idea behind this definition is that a preordered action
$(\alpha, \xi)$ of $X$ on $Y$ determines a preorder on $X \times Y$
that turns $X \rtimes_\alpha Y$ into a preordered monoid by specifying
its positive cone. More precisely, the function $\xi$ specifies the
cone of positive elements of $X \rtimes_\alpha Y$ by asserting that
$(x, y)$ is positive if, and only if, $\xi(x, y) = x$. The fact that
the domain of $\xi$ is $X \times P_Y$ yields that all positive
elements of $X \rtimes_\alpha Y$ belong to this set. Moreover, any
value taken by $\xi$ on a point $(x, y)$ that is different from~$x$ is
somehow irrelevant. Indeed, the authors define morphisms of preordered
actions as follows: if $X, X'$ and $Y, Y'$ are preordered monoids, and
$(\alpha, \xi)$ and $(\alpha', \xi')$ are preordered actions of $Y$ on
$X$ and of $Y'$ on $X'$, respectively, then a morphism from $(\alpha,
\xi)$ to $(\alpha', \xi')$ is pair $(f, g)$ such that $f: X \to X'$
and $g: Y \to Y'$ are monoid homomorphisms restricting and
co-restricting to the suitable positive cones, and that satisfy
\[f(\alpha(y, x)) = \alpha'(g(y), f(x)) \text{ and }\xi'(f(u), g(v)) =
  f(u),\]
for all $(x, y) \in X \times Y$, and $(u, v) \in X \times P_Y$ such
that $\xi(u, v) = u$. Therefore, if $(\alpha, \xi)$ and $(\alpha,
\xi')$ are preordered actions of $Y$ on $X$ such that, for all $x \in
X$ and $y \in P_Y$
\begin{equation}
  \xi(x, y) = x \iff \xi'(x, y) = x,\label{eq:12}
\end{equation}
then the pair $({\rm id}_X, {\rm id}_Y)$ consisting of the suitable
identity maps defines an isomorphism between $(\alpha, \xi)$ and
$(\alpha, \xi')$. We may thus identify two preordered actions
$(\alpha, \xi)$ and $(\alpha, \xi')$ whenever they
satisfy~\eqref{eq:12}.

We now compare the notions of $\two$-enriched action as in
Proposition~\ref{p:5} and preordered action as in
Definition~\ref{sec:preaction}.

\begin{proposition}
  Let $X$ and $Y$ belong to $\twomon^*$. Then, there is a one-to-one
  correspondence between $\two$-enriched actions of $Y$ on $X$ and
  preordered actions of $Y$ on $X$ determined by those $(\alpha, \xi)$
  that further satisfy
  \begin{enumerate}[label = (A.\arabic*)]\setcounter{enumi}{-1}
  \item \label{item:A0} if $\xi(x, 1) = x$, then $x \in P_X$, for all
    $x \in X$.
  \end{enumerate}
\end{proposition}
\begin{proof}
  The correspondence is as follows. If $(\alpha, P)$ is a
  $\two$-enriched action then $(\alpha, \xi)$ is a preordered action
  satisfying~\ref{item:A0}, where
  \[\xi(x, y) =
    \begin{cases}
      x, \text{ if $(x, y) \in P$,}
     \\ 0, \text{ else.}
    \end{cases}
  \]
  Conversely, if $(\alpha, \xi)$ is a preordered action that
  satisfies~\ref{item:A0} then, taking
  \[P = \{(x, y) \in X \times P_Y \mid \xi(x, y) = x\}\]
  defines a $\two$-enriched action $(\alpha, P)$. It is a routine
  computation to check that these two assignments are indeed
  well-defined and are mutually inverse. We highlight that, if we
  start with a preordered action $(\alpha, \xi)$ and $(\alpha, P)$ it
  the $\two$-enriched action it defines, then the preordered
  action $(\alpha, \xi')$ defined by $(\alpha, P)$ may not coincide
  with $(\alpha, \xi)$, but $(\alpha, \xi)$ and $(\alpha, \xi')$ do
  satisfy~\eqref{eq:12}.
\end{proof}
The main result of~\cite{Martins-FerreiraSobral21} states that, up to
isomorphism, there is a one-to-one correspondence between Schreier
split extensions of preordered monoids in $\vmon^*$ and preordered
actions. As we have just seen that the notions of \emph{preordered
  action} and of \emph{$\two$-enriched action} are slightly different,
the reader may now realize an apparent contradiction between this
result and our Theorem~\ref{t:2}. To understand what is happening, we
have to analyze the definition of Schreier split extension
of~\cite{Martins-FerreiraSobral21}. It is then the moment to introduce
yet a new category, which turns out to be isomorphic to $\bf
OrdMon^*$~\cite[Theorem~1]{Martins-FerreiraSobral21}.

\begin{definition}{\cite[Definition~2]{Martins-FerreiraSobral21}}
  A monomorphism of monoids $m: P \rightarrowtail M$ is \emph{right
    normal} if its image is a right normal submonoid of~$M$. The full
  subcategory of the category of monomorphisms of monoids is denoted
  by $\bf RNMono(Mon)$.
\end{definition}

The authors of~\cite{Martins-FerreiraSobral21} mostly work with the
category $\bf RNMono(Mon)$ rather than with $\bf OrdMon^*$, including
for defining Schreier split extensions.

\begin{definition}{\cite[Definition~3]{Martins-FerreiraSobral21}}
  A \emph{Schreier split epimorphism} in $\bf RNMono(Mon)$ is a
  diagram
  \begin{equation}\label{eq:5}
    \begin{tikzcd}
      P_X \arrow[rightarrowtail]{d} \arrow{r}{\overline{k}} & P_Z
      \arrow[rightarrowtail]{d} \arrow[yshift=3pt]{r}{\overline{p}} & P_Y
      \arrow[rightarrowtail]{d}  \arrow[yshift=-3pt]{l}{\overline{s}} \\
      X \arrow{r}{k} & Z \arrow[yshift=3pt]{r}{p} & Y
      \arrow[yshift=-3pt]{l}{s}
    \end{tikzcd}
  \end{equation}
  in which the lower row is a Schreier split epimorphism of monoids
  and the upper row consists if right normal submonoids, the positive
  cones $P_X$, $P_Z$, and $P_Y$, that turn $X$, $Z$, and $Y$ objects
  of $\bf OrdMon^*$. The morphisms $\overline{k}$, $\overline{p}$, and
  $\overline{s}$ are the corresponding restrictions.
\end{definition}

This definition obfuscates the real nature of the morphisms in $\bf
OrdMon^*$. Indeed, while a monoid homomorphism between preordered
monoids in $\bf OrdMon^*$ is monotone if, and only if, it preserves
the positive cone, the morphism $k$ in diagram~\eqref{eq:5}, although
monotone, is not the kernel of~$p$ in the category of preordered
monoids (nor in $\bf OrdMon^*$). If we want to ensure that the
Schreier split extension defined by a preordered action $(\alpha,
\xi)$ is such that $k$ is a kernel in the category of preordered
monoids, then we must require that $(\alpha, \xi)$ satisfies
\ref{item:A0}. Moreover, even in the case of $\cV$-groups, unlike
ours, the results of~\cite{Martins-FerreiraSobral21} are not
comparable with those
of~\cite{ClementinoMartins-FerreiraMontoli19,ClementinoRuivo2023,
  ClementinoMontoli21}, as shown by the next example.
\begin{example}
  Let $X$ and $Y$ be preordered groups and $\alpha: Y \times X \to X$
  be a group action. We consider the preorder~$\preceq$ on $X \times
  Y$ defined by
  \[(x, y) \preceq (x', y') \iff y \leq_Y y'.\]
  Then, $(X \rtimes_\alpha Y, \preceq)$ is a preordered group and
  \[ \begin{tikzcd} P_X \arrow[rightarrowtail]{d}
      \arrow{r}{\overline{\iota_1}} & P_{X \rtimes_\alpha Y}
      \arrow[rightarrowtail]{d}
      \arrow[yshift=3pt]{r}{\overline{\pi}_2} & P_Y
      \arrow[rightarrowtail]{d}  \arrow[yshift=-3pt]{l}{\overline{\iota}_2} \\
      X \arrow{r}{\iota_1} & X \rtimes_\alpha Y
      \arrow[yshift=3pt]{r}{\pi_2} & Y \arrow[yshift=-3pt]{l}{\iota_2}
    \end{tikzcd}\]
  is a Schreier split extension in $\bf RNMono(Mon)$. However, unless
  the preorder on~$X$ is trivial,
  \[X\overset{\iota_1} {\hookrightarrow} X \rtimes_\alpha Y
    \overset{\pi_2}{\underset{\iota_2}\rightleftarrows} Y\]
  is not a Schreier split extension of preordered groups in the sense
  of~\cite{ClementinoMartins-FerreiraMontoli19,
    ClementinoMontoli21}. Indeed, for all $x, x' \in X$ we have $(x,
  1) \preceq (x', 1)$. Thus, if there are some $x \nleq_X x'$, then
  $\iota_1[X]$ is not a submonoid of $X \rtimes_\alpha Y$, and
  therefore $k$ is not the kernel of $\pi_2$.
\end{example}
\footnotesize
\bibliographystyle{plain}

\end{document}